\documentclass[12pt,a4paper]{amsart}

\usepackage{etex}
\usepackage[utf8]{inputenc}
\usepackage{amsmath}
\usepackage[all]{xy}
\usepackage{latexsym}
\usepackage[usenames,dvipsnames,svgnames,table]{xcolor}
\usepackage{amsfonts}
\usepackage{amssymb}
\usepackage{amsthm}
\usepackage{enumerate}
\usepackage{mathtools}
\usepackage[vcentermath]{youngtab} 
\usepackage{ytableau}
\usepackage{bm}
\usepackage[makeroom]{cancel}
\usepackage{float}
\usepackage{paralist}
\usepackage{graphicx}
\usepackage{pifont}
\usepackage{hyperref}
\usepackage{caption}
\usepackage{booktabs}
\usepackage{multicol}
\usepackage{tikz}
\usetikzlibrary{arrows, fit, automata}
\usetikzlibrary{calc,positioning,decorations.pathreplacing,decorations.text,intersections}

\newcommand\abs[1]{\left|#1\right|}
\newcommand*{\aA}{\mathcal{A}}
\newcommand*{\pPl}{{\smash{\mathrm{P}}}_{\smash{\mathrm{lps}}}}
\newcommand*{\pPr}{{\smash{\mathrm{P}}}_{\smash{\mathrm{rps}}}}
\newcommand*{\pPx}{{\smash{\mathrm{P}}}_{\smash{\mathrm{xps}}}}

\newcommand*{\plac}{{\smash{\mathrm{plac}}}}
\newcommand*{\bell}{{\smash{\mathrm{lps}}}}

\newcommand*{\belr}{{\smash{\mathrm{rps}}}}
\newcommand*{\belx}{{\smash{\mathrm{xps}}}}
\newcommand*{\bellcong}{\equiv_\bell}
\newcommand*{\belrcong}{\equiv_\belr}
\newcommand*{\belxcong}{\equiv_\belx}
\newcommand*{\rR}{\mathcal{R}}

\newcommand*{\cont}{{\smash{\mathrm{cont}}}}
\newcommand*{\coch}{{\smash{\mathrm{cochseq}}}}

\newcommand*{\K}{{\smash{\mathrm{K}}}}

\def \lps {lPS}
\def \rps {rPS}
\def \ps {PS}

\DeclarePairedDelimiter{\parens}{\lparen}{\rparen}
\newcommand*{\evlit}{{\mathrm{ev}}}
\newcommand*{\plit}{{\mathrm{p}}}
\newcommand*{\llit}{{\mathrm{l}}}
\newcommand*{\olit}{{\mathrm{o}}}
\newcommand*{\rlit}{{\mathrm{r}}}

\newcommand*{\ev}[2][]{\evlit\parens[#1]{#2}}

\newcommand*{\evsim}{\sim_\evlit}
\newcommand*{\psim}{\sim_\plit}
\newcommand*{\tpsim}{{\sim}_\plit^*}
\newcommand*{\lsim}{\sim_\llit}
\newcommand*{\osim}{\sim_\olit}
\newcommand*{\rsim}{\sim_\rlit}
\newcommand*{\nlsim}{\nsim_\llit}
\newcommand*{\npsim}{\nsim_\plit}

\theoremstyle{plain}
\newtheorem{thm}{Theorem}[section]
\newtheorem{lem}[thm]{Lemma}
\newtheorem{prop}[thm]{Proposition}

\theoremstyle{definition}

\newtheorem{algorithm}[thm]{Algorithm}
\newtheorem{conj}[thm]{Conjecture}

\newtheorem{prob}[thm]{Open Problem}

\makeatletter

\pgfkeys{/pgf/.cd,
	rectangle corner radius/.initial=3pt
}
\newif\ifpgf@rectanglewrc@donecorner@
\def\pgf@rectanglewithroundedcorners@docorner#1#2#3#4{%
	\edef\pgf@marshal{%
		\noexpand\pgfintersectionofpaths
		{%
			\noexpand\pgfpathmoveto{\noexpand\pgfpoint{\the\pgf@xa}{\the\pgf@ya}}%
			\noexpand\pgfpathlineto{\noexpand\pgfpoint{\the\pgf@x}{\the\pgf@y}}%
		}%
		{%
			\noexpand\pgfpathmoveto{\noexpand\pgfpointadd
				{\noexpand\pgfpoint{\the\pgf@xc}{\the\pgf@yc}}%
				{\noexpand\pgfpoint{#1}{#2}}}%
			\noexpand\pgfpatharc{#3}{#4}{\cornerradius}%
		}%
	}%
	\pgf@process{\pgf@marshal\pgfpointintersectionsolution{1}}%
	\pgf@process{\pgftransforminvert\pgfpointtransformed{}}%
	\pgf@rectanglewrc@donecorner@true
}
\pgfdeclareshape{rectangle with rounded corners}
{
	\inheritsavedanchors[from=rectangle] 
	\inheritanchor[from=rectangle]{north}
	\inheritanchor[from=rectangle]{north west}
	\inheritanchor[from=rectangle]{north east}
	\inheritanchor[from=rectangle]{center}
	\inheritanchor[from=rectangle]{west}
	\inheritanchor[from=rectangle]{east}
	\inheritanchor[from=rectangle]{mid}
	\inheritanchor[from=rectangle]{mid west}
	\inheritanchor[from=rectangle]{mid east}
	\inheritanchor[from=rectangle]{base}
	\inheritanchor[from=rectangle]{base west}
	\inheritanchor[from=rectangle]{base east}
	\inheritanchor[from=rectangle]{south}
	\inheritanchor[from=rectangle]{south west}
	\inheritanchor[from=rectangle]{south east}

	\savedmacro\cornerradius{%
		\edef\cornerradius{\pgfkeysvalueof{/pgf/rectangle corner radius}}%
	}

	\backgroundpath{%
		\northeast\advance\pgf@y-\cornerradius\relax
		\pgfpathmoveto{}%
		\pgfpatharc{0}{90}{\cornerradius}%
		\northeast\pgf@ya=\pgf@y\southwest\advance\pgf@x\cornerradius\relax\pgf@y=\pgf@ya
		\pgfpathlineto{}%
		\pgfpatharc{90}{180}{\cornerradius}%
		\southwest\advance\pgf@y\cornerradius\relax
		\pgfpathlineto{}%
		\pgfpatharc{180}{270}{\cornerradius}%
		\northeast\pgf@xa=\pgf@x\advance\pgf@xa-\cornerradius\southwest\pgf@x=\pgf@xa
		\pgfpathlineto{}%
		\pgfpatharc{270}{360}{\cornerradius}%
		\northeast\advance\pgf@y-\cornerradius\relax
		\pgfpathlineto{}%
	}

	\anchor{before north east}{\northeast\advance\pgf@y-\cornerradius}
	\anchor{after north east}{\northeast\advance\pgf@x-\cornerradius}
	\anchor{before north
		west}{\southwest\pgf@xa=\pgf@x\advance\pgf@xa\cornerradius
		\northeast\pgf@x=\pgf@xa}
	\anchor{after north
		west}{\northeast\pgf@ya=\pgf@y\advance\pgf@ya-\cornerradius
		\southwest\pgf@y=\pgf@ya}
	\anchor{before south west}{\southwest\advance\pgf@y\cornerradius}
	\anchor{after south west}{\southwest\advance\pgf@x\cornerradius}
	\anchor{before south
		east}{\northeast\pgf@xa=\pgf@x\advance\pgf@xa-\cornerradius
		\southwest\pgf@x=\pgf@xa}
	\anchor{after south east}{\southwest\pgf@ya=\pgf@y\advance\pgf@ya\cornerradius
		\northeast\pgf@y=\pgf@ya}

	\anchorborder{%
		\pgf@xb=\pgf@x
		\pgf@yb=\pgf@y%
		\southwest%
		\pgf@xa=\pgf@x
		\pgf@ya=\pgf@y%
		\northeast%
		\advance\pgf@x by-\pgf@xa%
		\advance\pgf@y by-\pgf@ya%
		\pgf@xc=.5\pgf@x
		\pgf@yc=.5\pgf@y%
		\advance\pgf@xa by\pgf@xc
		\advance\pgf@ya by\pgf@yc%
		\edef\pgf@marshal{%
			\noexpand\pgfpointborderrectangle
			{\noexpand\pgfqpoint{\the\pgf@xb}{\the\pgf@yb}}
			{\noexpand\pgfqpoint{\the\pgf@xc}{\the\pgf@yc}}%
		}%
		\pgf@process{\pgf@marshal}%
		\advance\pgf@x by\pgf@xa%
		\advance\pgf@y by\pgf@ya%
		\pgfextract@process\borderpoint{}%
		\pgf@rectanglewrc@donecorner@false
		\southwest\pgf@xc=\pgf@x\pgf@yc=\pgf@y
		\advance\pgf@xc\cornerradius\relax\advance\pgf@yc\cornerradius\relax
		\borderpoint
		\ifdim\pgf@x<\pgf@xc\relax\ifdim\pgf@y<\pgf@yc\relax
		\pgf@rectanglewithroundedcorners@docorner{-\cornerradius}{0pt}{180}{270}%
		\fi\fi
		\ifpgf@rectanglewrc@donecorner@\else
		\southwest\pgf@yc=\pgf@y\relax\northeast\pgf@xc=\pgf@x\relax
		\advance\pgf@xc-\cornerradius\relax\advance\pgf@yc\cornerradius\relax
		\borderpoint
		\ifdim\pgf@x>\pgf@xc\relax\ifdim\pgf@y<\pgf@yc\relax
		\pgf@rectanglewithroundedcorners@docorner{0pt}{-\cornerradius}{270}{360}%
		\fi\fi
		\fi
		\ifpgf@rectanglewrc@donecorner@\else
		\northeast\pgf@xc=\pgf@x\relax\pgf@yc=\pgf@y\relax
		\advance\pgf@xc-\cornerradius\relax\advance\pgf@yc-\cornerradius\relax
		\borderpoint
		\ifdim\pgf@x>\pgf@xc\relax\ifdim\pgf@y>\pgf@yc\relax
		\pgf@rectanglewithroundedcorners@docorner{\cornerradius}{0pt}{0}{90}%
		\fi\fi
		\fi
		\ifpgf@rectanglewrc@donecorner@\else
		\northeast\pgf@yc=\pgf@y\relax\southwest\pgf@xc=\pgf@x\relax
		\advance\pgf@xc\cornerradius\relax\advance\pgf@yc-\cornerradius\relax
		\borderpoint
		\ifdim\pgf@x<\pgf@xc\relax\ifdim\pgf@y>\pgf@yc\relax
		\pgf@rectanglewithroundedcorners@docorner{0pt}{\cornerradius}{90}{180}%
		\fi\fi
		\fi
	}
}

\makeatother

\makeindex

\title{Conjugacy in Patience Sorting monoids}

\author{Alan J. Cain}
\address{Centro de Matem\'atica e Aplica\c{c}\~oes, Faculdade de Ci\^encias e
Tecnologia, Universidade Nova de Lisboa,
	2829-516 Caparica, Portugal}
\email{a.cain@fct.unl.pt}
\thanks{The first author was supported by
an Investigador FCT fellowship (IF/01622/2013/CP1161/CT0001).}

\author{Ant\'onio
Malheiro}
\address{Centro de Matem\'atica e Aplica\c{c}\~oes and Departamento de
	Matem\'atica, Faculdade de Ci\^encias e Tecnologia, Universidade Nova de
Lisboa,
	2829-516 Caparica, Portugal}
\email{ajm@fct.unl.pt}
\thanks{For the first two authors,
this work was partially supported by the Funda\c{c}\~ao para a
Ci\^encia e a Tecnologia (Portuguese Foundation for Science and Technology)
through the project UID/MAT/00297/2013 (Centro de Matem\'atica e
Aplica\c{c}\~oes), and the project PTDC/MHC-FIL/2583/2014}

\author{F\'abio M. Silva}
\address{Departamento de Matem\'atica
	and CEMAT-CI\^ENCIAS, Faculdade de Ci\^encias, Universidade de Lisboa,
Lisboa
	1749-016, Portugal.}
\email{femsilva@fc.ul.pt}
\thanks{The third author was supported by an FCT Lismath fellowship
(PD/BD/52644/2014) and partially supported by the FCT project CEMAT-Ci\^encias
UID/Multi/04621/2013}

\begin{document}

\begin{abstract}
The cyclic shift graph of a monoid is the graph whose vertices are the elements
of the monoid and whose edges connect elements that are cyclic shift related.
The Patience Sorting algorithm admits two generalizations to words, from which
two kinds of monoids arise, the $\belr$ monoid and the $\bell$ (also known as
Bell) monoid. Like other monoids arising from combinatorial objects such as
the plactic and the sylvester, the connected components of the cyclic shift
graph of the $\belr$ monoid consists of elements that have the same number of
each of its composing symbols. In this paper, with the aid of the computational
tool SageMath, we study the diameter of the connected components from the cyclic
shift graph of the $\belr$ monoid.

Within the theory of monoids, the cyclic shift relation, among other relations,
generalizes the relation of conjugacy for groups. We examine several of these
relations for both the $\belr$ and the $\bell$ monoids.

\end{abstract}


\maketitle

\section{Introduction}
\label{sec:introduction}

Patience Sorting has its origins in the works of Mallows
\cite{Mallows62,10.2307/2028347} and can be regarded as an insertion algorithm
on standard words over a totally ordered alphabet $\aA_n=\{1<2<\dots<n\}$,
that is, words over $\aA_n$ containing exactly one occurrence of each of the
symbols from $\aA_n$. As noticed by Burstein and Lankham \cite{BL2005}, this
algorithm can be viewed as a non-recursive version
of Schensted's insertion algorithm. This perspective suggests that a construction
similar to the plactic monoid must also hold for this case. The plactic monoid
can be constructed as the quotient of the free monoid over $\aA$ (the infinite
totally ordered alphabet of natural numbers),
$\aA^*$, by the congruence which relates words of $\aA^*$ inserting to the
same (semistandard) Young tableaux under Schensted's insertion algorithm.

According to Aldous
and Diaconis \cite{MR1694204} we can consider two generalizations of Patience
Sorting to words, which we will call the right Patience Sorting insertion and
the left Patience Sorting insertion (\rps\ and \lps\ insertion, respectively,
for short). Considering the alphabet $\aA$, these generalizations lead to
two distinct monoids, the \rps\ monoid, denoted by $\belr$,
and the \lps\ monoid (also known in the literature as the Bell monoid
\cite{Maxime07}), denoted by $\bell$, which are,
respectively, the monoids given by the quotient of $\aA^*$ by the congruence
which relates words having the same insertion under the \rps\ and \lps\
insertion.

In a monoid $M$, two elements $u$ and $v$, are said to be related by a
cyclic  shift, denoted $u\psim v$, if there exists $x,y\in M$ such that
$u=xy$ and $v=yx$. In their seminal work concerning the plactic monoid
\cite{MR646486}, Lascoux and Sch\"utzenberger proved that 	 any two
elements in the plactic monoid, $\plac$,
having the same evaluation  (that is, elements that contain the
same number of each generating symbol) can be obtained one from the other by
applying
a finite sequence of cyclic shift relations. The same characterization is known
to hold for other plactic-like monoids, such as the hypoplactic monoid
\cite{1709.03974},  the Chinese monoid \cite{MR1847182},  the sylvester
monoid \cite{MR2081336,MR2142078}, and the taiga monoid \cite{1709.03974}.
In Section~\ref{subsection:conjugacy} we show
that an analogous result  holds for the \rps\ monoids (of finite rank)
and for the \lps\ monoid of rank $1$, $\bell_1$.
Note that all these monoids are multihomogeneous, that is, they are defined by
presentations where the two side of each defining relation contains the same
number of each generator. Thus,  the evaluation of an element of the monoid
corresponds to the evaluation of some (and hence any) word that represents it.

The previous results can be rewritten in another form by considering what we
will
call as cyclic shift graph of a monoid $M$, denoted $\K(M)$, which is the
undirected graph whose vertices are the elements of $M$ and whose edges
connect elements that differ by a cyclic shift. So, if $M=\plac$
or $M=\belr$, or their finite analogues, then the results mentioned in the
previous paragraph can
be restated as saying that the connected components of $\K(M)$ consist of
the elements of $M$ which have the same evaluation. Thus, it follows that
the connected components of $\K(M)$ are finite. With the aid of the
computational tool SageMath we studied the diameter of the connected
components of the cyclic shift graph $\K(\belr)$. In SageMath we wrote
a program based on the \rps\ insertion algorithm, which given a word of
$\aA^*$, outputs 
the connected component of $\K(\belr)$ containing the element of $\belr$
that corresponds to the evaluation of the inserted word .

Aiming to parallel the result obtained by Choffrut
and Merca{\c{s}} \cite{Choffrut2013}, and refined by Cain and Malheiro
\cite{1709.03974}, concerning the maximal diameter of connected components of
the cyclic shift graph of the plactic monoid of finite rank, we used the tools
available in the SageMath library
to construct tables containing the number of vertices and the diameter
of connected components from $\K(\belr_n)$.
The experimental results obtained from these calculations lead us to establish
some conjectures regarding diameters of specific connected components.
In Section~\ref{subsection:cyclic_shift}, we show that some of these
conjectures are in fact true.  In
particular we prove that the maximum diameter of a connected components
of $\K(\belr_n)$, for $n\geq 3$, lies   between $n-1$ and $2n-4$. We
also draw some conclusions for the diameter of $\K(\belr_n)$ for  particular
elements of $\belr_n$.

The cyclic shift relation previously defined generalizes the usual conjugacy
relation for groups. That is, when considering groups, the cyclic shift
relation is just the usual conjugacy relation. Since for monoids this
relation is, in general, not transitive, it is natural to consider the
transitive
closure of $\psim$, which we will henceforth denote by $\tpsim$. (Note that
$\tpsim$-classes correspond to connected components of the cyclic shift graph.)
We consider two other notions of  conjugacy (see \cite{Araujo201493,
araujo2015four} for other conjugacy  notions, their properties, and
relations among them). The relation $\lsim$
on $M$, proposed by Lallement in \cite{MR530552}, which can be defined as
follows: given $u,v\in M$
\begin{equation*}
u\lsim v\ \Leftrightarrow\ \exists g\in M\ ug=gv.
\end{equation*}
(There is a dual notion $\rsim$ relating elements for which $gu=vg$,
instead.) As this relation is reflexive and transitive but, in general, not
symmetric, in \cite{MR742135}, Otto considered the equivalence
relation $\osim$ given by the intersection of $\lsim$ and $\rsim$.

All the mentioned relations are equal in the group case, and in any
monoid, $\psim\ \subseteq\ \tpsim\ \subseteq\ \osim\
\subseteq\ \lsim$ (cf. \cite{Araujo201493}).
Denoting by $\evsim$ the binary relation that pairs elements with the
same evaluation, it is easy to see that for multihomogeneous monoids $\lsim\
\subseteq\ \evsim $ (cf. \cite[Lemma~3.2]{cm_conjugacy}), and thus for all the
above multihomogeneous monoids (plactic, hypoplactic, chinese, sylvester, taiga
and \rps)  we have
${\tpsim} = {\osim} = {\lsim}  = {\evsim}$. This property, is not a general
property of multihomogeneous monoids, as it is known that in the stalactic
monoid
connected components of the cyclic shift graph are properly contained in
$\evsim$ \cite[Proposition~7.2]{1709.03974}. In this paper we show that a
similar situation occurs for \lps\ monoids of rank greater than 1, since we
will prove that ${\lsim} \subsetneq {\evsim}$ in these cases.

\section{Preliminaries and notation}

In this section we introduce the fundamental notions that we will use along the
paper. For more details regarding these concepts check for instance
\cite{1706.06884}, \cite{MR1905123}, and \cite{howie1995fundamentals}.

\subsection{Words and presentations}
\label{alphabetswords}

In this paper, we denote by $\aA$ the infinite totally ordered alphabet
$\{1<2<\dots \}$, that is, the set of natural numbers with the usual order
viewed as an alphabet. For any $n\in\mathbb{N}$, the resriction of $\aA$ to the
first $n$ natural numbers is denoted by $\aA_n$.

In general, if $\Sigma$ is an alphabet, then
$\Sigma^+$ denotes the \emph{free semigroup} over $\Sigma$, that is,
the set of non-empty words over $\Sigma$, and if $\varepsilon$ denotes
the empty word, then the \emph{free monoid} over $\Sigma$ is $\Sigma^*=
\Sigma^+\cup \{\varepsilon\}$.

Next, we  define several concepts that are
directly related with the notion of word.
Let $w\in\aA^*$. Then:
\begin{itemize}
	\item a word $u\in \aA^*$,  is said to be a \emph{factor} of $w$ if
	there exist words $v_1,v_2\in \aA^*$, such that
	$w=v_1uv_2$;
	\item for any symbol $a$ in $\aA$, the number of occurrences of
$a$ in
	$w$, is denoted by $\abs{w}_{a}$;
	\item the \emph{content of} $w$, is the set $\cont(w)=\left\{{a}\in
	\aA: \abs{w}_{a}\geq 1\right\}$;
	\item the \emph{evaluation of} $w$, denoted by $\ev{w}$, is the
	sequence of non-negative integers
	whose $a$-th term is
	$\abs{w}_{a}$, for any $a\in \aA$;
	\item the word is said to be \emph{standard} if each symbol from
$\aA_n$, for a given $n$, occurs exactly once.
\end{itemize}

A \emph{monoid presentation} is a pair $(\Sigma,
\mathcal{R})$, where $\Sigma$ is an alphabet and $\mathcal{R}\subseteq
\Sigma^*\times \Sigma^*$. We say that a monoid $M$ is \emph{defined by a
presentation} $(\Sigma,\mathcal{R})$ if $M\simeq \Sigma^*/\mathcal{R}^\#$,
where $\mathcal{R}^\#$ is the smallest congruence containing $\mathcal{R}$
(see \cite[Proposition~1.5.9]{howie1995fundamentals} for a combinatorial
description of the smallest congruence containing a relation).

A presentation is \emph{multihomogeneous} if, for every relation
$(w,w') \in\mathcal{R}$, we have $\ev{w}=\ev{w'}$, in other words, if  $w$ and
$w'$ contain the same number of each of
its composing symbols. Then, a monoid is multihomogeneous if
there exists a multihomogeneous presentation defining the monoid.


\subsection{\ps\ \textit{tableaux} and insertion}

In this subsection we recall the basic concepts regarding patience sorting
\emph{tableaux}, and the insertion on such \emph{tableaux}.

A \emph{composition diagram} is a finite collection of boxes arranged in
bottom-justified columns, where no order on the length of the columns is
imposed. Let $\Sigma$
be a totally ordered alphabet. Then, an \lps\ (resp. \rps)
\emph{tableau over}  $\Sigma$ is a composition diagram with entries
from $\Sigma$, so that the sequence of entries of the boxes in each column
is strictly (resp., weakly) decreasing from top to bottom, and the sequence
of entries of the boxes in the bottom row is weakly (resp., strictly)
increasing from left to right. So, if
\begin{equation} \label{exmp2}
\ytableausetup
{aligntableaux=center, boxsize=1.25em}
R=\begin{ytableau}
\none &\none  & 4 \\
4 & 5 & 3 \\
1 & 1 & 2
\end{ytableau}\ \text{ and }\  S=\begin{ytableau}
\none & 5 \\
4 & 4   \\
1 & 3 \\
1 & 2
\end{ytableau},
\end{equation}
then $R$ is an \lps\ \emph{tableau}, and $S$ is an \rps\ \emph{tableau} both
over $\aA_n$, for $n \geq 5$. Henceforth,
we shall often refer to an \lps\ tableau or to an \rps\ tableau simply
as a \ps\ tableau, not distinguishing the cases whenever they can be dealt
in a similar way.

The left and right Patience Sorting monoids  can be
given as the quotient of the free monoid $\aA^*$ over the congruence
which relates words that yield the same \ps\ tableau under a certain algorithm
\cite[\S~3.6]{1706.06884}.
This algorithm is presented in the following paragraph and merges
in one the Algorithms 3.1 and 3.2 of \cite{1706.06884}. (Observe that
we will use the notation $\pPl()$, $\pPr()$ instead of, respectively,
$\mathfrak{R}_\ell()$, $\mathfrak{R}_r()$ used in \cite{1706.06884}.)

\begin{algorithm}[\ps\ insertion of a word]\label{alg:PSinsertion}
	~\par\nobreak
	\textit{Input:} A word $w$ over a totally ordered alphabet $\Sigma$.

	\textit{Output:} An \lps\ \emph{tableau} $\pPl(w)$ (resp., \rps\
	\emph{tableau} $\pPr(w)$).

	\textit{Method:}
	\begin{enumerate}
		\item If $w=\varepsilon$, output an empty \textit{tableau}
		$\emptyset$. Otherwise:
		\item $w=w_1\cdots w_n$, with $w_1,\ldots,w_n\in \Sigma$. Setting
		\begin{align*}
		\pPl(w_1)=\ytableausetup
		{boxsize=1.1em, aligntableaux=center}\begin{ytableau}
		w_1 \end{ytableau} =\pPr(w_1),
		\end{align*}
		then, for each remaining symbol $w_j$ with $1<j\leq n$,
		denoting by $r_1\leq \dots\leq r_k$ (resp., $r_1< \dots< r_k$)
the
		symbols in the bottom row of the \textit{tableau} $\pPl(w_1\cdots w_{j-1})$
		(resp., $\pPr(w_1\cdots w_{j-1})$), proceed as follows:
		\begin{itemize}
			\item if $r_k\leq w_j$ (resp., $r_k < w_j$),
insert $w_j$ in a new
			column to the right of $r_k$ in $\pPl(w_1\cdots w_{j-1})$ (resp., $\pPr(
			w_1\cdots w_{j-1})$);

			\item otherwise, if $m=\min\left\{i\in\{1,\ldots,
k\}:w_j< r_i\right\}$, (resp. $m=\min\left\{i\in\{1,\ldots,
k\}:w_j\leq r_i\right\}$)
			construct a new empty box on top of the column of
			$\pPl(w_1\cdots w_{j-1})$ (resp. $\pPr(w_1\cdots w_{j-1})$) containing
			$r_m$. Then bump all the symbols of the column containing $r_m$
			to the box above and insert $w_j$ in the box which has been cleared
			and previously contained the symbol $r_m$.
		\end{itemize}
		Output the resulting \emph{tableau}.
	\end{enumerate}
\end{algorithm}

Observe that the insertion of a given word $w=w_1\cdots w_n$
under Algorithm~\ref{alg:PSinsertion} is obtained through the
insertion of each of its symbols, from left to right in the
previously obtained tableaux (starting with the empty tableaux
$\emptyset$).
For instance, if $R$ is the tableau from Example~\ref{exmp2}, and $u=4511432
\in\aA^*_5$, then  $\pPl(u)=R$ (see Figure~\ref{figure:extended_insertion}).
The reader can check that $\pPr(u)=S$.

\begin{figure}[!htp]
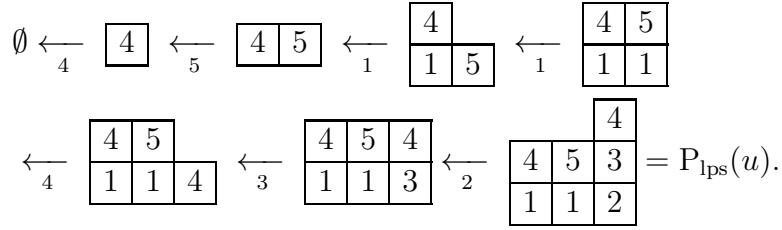

	\centering
	\begin{align*}
	&\emptyset \xleftarrow[4]{\ \ \ }\
	\ytableausetup
	{mathmode, boxsize=1.3em, aligntableaux=center}
	\begin{ytableau}
	4
	\end{ytableau}\
	\xleftarrow[5]{\ \ \ }\
	\begin{ytableau}
	4 & 5
	\end{ytableau}\
	\xleftarrow[1]{\ \ \ }\
	\begin{ytableau}
	4 & \none\\
	1 & 5
	\end{ytableau}\
	\xleftarrow[1]{\ \ \ }\
	\begin{ytableau}
	4 & 5\\
	1 & 1
	\end{ytableau}\\
	&\xleftarrow[4]{\ \ \ }\
	\begin{ytableau}
	4 & 5 & \none \\
	1 & 1 & 4
	\end{ytableau}\
	\xleftarrow[3]{\ \ \ }\
	\begin{ytableau}
	4 & 5 & 4 \\
	1 & 1 & 3
	\end{ytableau}
	\xleftarrow[2]{\ \ \ }\
	\begin{ytableau}
	\none & \none & 4\\
	4 & 5 & 3\\
	1 & 1 & 2
	\end{ytableau}=\pPl(u ).\
	\end{align*}
	\caption{\lps\ insertion of the word $w=4511432$,
		where the symbol below the arrow indicates the symbol
		that is going to be inserted on each step.}
	\label{figure:extended_insertion}
\end{figure}

\subsection{The Patience Sorting monoids}

For each $\mathrm{x}\in \{\llit,\rlit\}$, we define a binary relation
$\belxcong$ in $\aA^*$ in the following way: given $u,v\in\aA^*$,
\[ u\belxcong v\quad \textrm{iff}\quad  \pPx(u)=\pPx(v).
\]
This relation is a congruence \cite[Proposition~3.21]{1706.06884}, and the
quotient of $\aA^*$ by $\bellcong$ is the so-called \lps\ monoid, denoted
$\bell$, and the quotient of $\aA^*$ by $\belrcong$ is the \rps\ monoid which
is denoted by $\belr$. The  rank-$n$ analogues of these monoids, denoted by
$\bell_n$ and $\belr_n$, are obtained by restricting the alphabet and the
relation to the set $\aA_n^*$. Note that each equivalence class of these
monoids is represented by a unique tableau, and hence we will identify elements
of the monoid with their tableaux representation.

Words yielding the same \ps\ tableau (and hence in the same $\belxcong$-class)
have necessarily the same content, and even the same evaluation. Thus, we can
refer to the content and evaluation of an element of the monoid, and similarly
to the content and  evaluation of a tableau. Also, we shall refer to an
 element of $\belx_n$ (or to its tableau representative) as \emph{standard} if
one (and hence any) of its words in the $\belxcong$-class has
one occurrence of each of the symbols from $\aA_n$.

As shown in
\cite[\S~3.6~\&~\S~3.7]{1706.06884}, the left and right Patience Sorting
monoids are defined by the multihomogeneous
presentations $(\aA^*,\rR_\bell)$ and $(\aA^*,\rR_\belr)$, where
\begin{align*}
 \rR_\bell&=\{\,(yux,yxu): m\in \mathbb{N},\, x,y,u_1,\ldots , u_m\in
\aA, \\
 &\qquad   u=u_m\cdots u_1,\, x<y\leq
u_1< \cdots < u_m\,\}
\end{align*}
and
\begin{align*}
\rR_\belr&=\{\,(yux,yxu): m\in \mathbb{N},\, x,y,u_1,\ldots , u_m\in
\aA, \\
&\qquad  u=u_m\cdots u_1,\, x\leq y<
u_1\leq \cdots \leq u_m\,\}.
\end{align*}
Hence, the left and right Patience Sorting
monoids, and their finite rank analogues,  are multihomogeneous
monoids.

We have seen how to obtain a \ps\ tableau from a word in $\aA^*$. Now,
we explain how to pass from \ps\ tableaux to words representing
such diagrams. Given $\textrm{x}\in \{\llit,\rlit\}$ and an $x$\ps\ tableau
$P$, the
\emph{column reading of} $P$ is the word obtained from reading the entries of
the $x$\ps\ tableau $P$, column by column, from the leftmost to the rightmost,
starting on the top of each column and ending on its bottom.
For example, the column reading of the \lps\ tableau $R$ in Example~\ref{exmp2}
is $41\, 51\, 432$, while the column reading of the \rps\ tableau $S$ is $411\,
5432$.

\section{Combinatorics of cyclic shifts}
\label{Subsection4.1}
\label{subsection:cyclic_shift}

As noted in the introduction, the \emph{cyclic
shift graph of} a monoid $M$, $\K(M)$, is the undirected graph with vertex set
$M$, whose edges connect vertices that differ by a single cyclic shift.
Since, $\belr$ is a multihomogeneous monoid, we have ${\tpsim}\subseteq
{\evsim}$, and thus each connected component of
$\K(\belr)$ cannot contain elements with different
	evaluations and therefore they have finitely many vertices.

Our goal in this subsection is to study the diameter of the connected
components from $\K(\belr_n)$, which as we will show are bounded by a value
that
depends on the rank $n$.
Note that in \cite[Example~3.1]{1709.03974}, the authors provide a finitely
presented multihomogeneous monoid for which the connected components of the
cyclic shift graph have unbounded diameter. Therefore, these are not
particular cases of a more general result that holds for all multihomogeneous
monoids.

The experimental results within this subsection were obtained with the aid of
SageMath \cite{cocalc}. This computational tool allowed us to write
a program for which: given an element of $\belr_n$, provides the connected
component from the cyclic shift graph of $\belr_n$ containing that element.

The program starts by creating a vertex for each word from $\aA^*_n$
that has the same evaluation as the given element from $\belr_n$.
Afterwards, it adds edges between the words that are
cyclic shift related. Finally, by merging the vertices whose $x$\ps\
insertion is the same into a single vertex,
it constructs the connected component of the cyclic shift graph of
$\belr_n$, $\K(\belr_n)$, containing the given element from $\belr_n$.

For instance in Figure~\ref{fig:connected_component} we show the
connected component of the cyclic shift graph of $\belr_4$ containing
the element $\pPr(1234)$ that can be seen to have diameter $4$.

\begin{figure}
	\centering
	\begin{tikzpicture}[auto, node distance=2.5cm, every loop/.style={}]
	\useasboundingbox (0.5,-3.5) -- (9.5,6);
	\node[rectangle with rounded corners, draw, inner sep=3pt, minimum size=17pt]	(1) {\color{Grey}\begin{ytableau}
		{\color{black}1} & {\color{black}2} & {\color{black}3} & {\color{black}4}
		\end{ytableau}};
	\node[rectangle with rounded corners, draw=Grey, inner sep=3pt, minimum size=17pt]	(2) [right of=1] {\color{Grey}\begin{ytableau}
		{\color{black}2} & \none & \none \\
		{\color{black}1} & {\color{black}3} & {\color{black}4}
		\end{ytableau}};
	\node[rectangle with rounded corners, draw=Grey, inner sep=3pt, minimum size=17pt]	(3) [above of=2] {\color{Grey}\begin{ytableau}
		{\color{black}3} & {\color{black}4} \\
		{\color{black}1} & {\color{black}2}
		\end{ytableau}};
	\node[rectangle with rounded corners, draw=Grey, inner sep=3pt, minimum size=17pt]	(4) [below of=2] {\color{Grey}\begin{ytableau}
		{\color{black}4} & \none & \none\\
		{\color{black}1} & {\color{black}2} & {\color{black}3}
		\end{ytableau}};
	\node[rectangle with rounded corners, draw=Grey, inner sep=3pt, minimum size=17pt]	(5) [right of=2] {\color{Grey}\begin{ytableau}
		{\color{black}3} & \none\\
		{\color{black}2} & \none\\
		{\color{black}1} & {\color{black}4}
		\end{ytableau}};
	\node[rectangle with rounded corners, draw=Grey, inner sep=3pt, minimum size=17pt]	(6) [above of=5] {\color{Grey}\begin{ytableau}
		{\color{black}4} & \none\\
		{\color{black}2} & \none\\
		{\color{black}1} & {\color{black}3}
		\end{ytableau}};
	\node[rectangle with rounded corners, draw=Grey, inner sep=3pt, minimum size=17pt]	(7) [above of=6] {\color{Grey}\begin{ytableau}
		\none & {\color{black}4} &\none\\
		{\color{black}1} & {\color{black}2} & {\color{black}3}
		\end{ytableau}};
	\node[rectangle with rounded corners, draw=Grey, inner sep=3pt, minimum size=17pt]	(8) [below of=5] {\color{Grey}\begin{ytableau}
		\none & {\color{black}3} & \none\\
		{\color{black}1} & {\color{black}2} & {\color{black}4}
		\end{ytableau}};
	\node[rectangle with rounded corners, draw=Grey, inner sep=3pt, minimum size=17pt]	(9) [right of=5] {\color{Grey}\begin{ytableau}
		{\color{black}2} & {\color{black}4}\\
		{\color{black}1} & {\color{black}3}
		\end{ytableau}};
	\node[rectangle with rounded corners, draw=Grey, inner sep=3pt, minimum size=17pt]	(10) [above of=9] {\color{Grey}\begin{ytableau}
		{\color{black}4}\\
		{\color{black}3}\\
		{\color{black}2}\\
		{\color{black}1}
		\end{ytableau}};
	\node[rectangle with rounded corners, draw=Grey, inner sep=3pt, minimum size=17pt]	(11) [above of=10] {\color{Grey}\begin{ytableau}
		\none & {\color{black}4}\\
		\none & {\color{black}3}\\
		{\color{black}1} & {\color{black}2}
		\end{ytableau}};
	\node[rectangle with rounded corners, draw=Grey, inner sep=3pt, minimum size=17pt]	(12) [below of=9] {\color{Grey}\begin{ytableau}
		{\color{black}4} & {\color{black}3}\\
		{\color{black}1} & {\color{black}2}
		\end{ytableau}};
	\node[rectangle with rounded corners, draw=Grey, inner sep=3pt, minimum size=17pt]	(13) [right of=9] {\color{Grey}\begin{ytableau}
		\none & \none & {\color{black}4}\\
		{\color{black}1} & {\color{black}2} & {\color{black}3}
		\end{ytableau}};
	\node[rectangle with rounded corners, draw=Grey, inner sep=3pt, minimum size=17pt]	(14) [above of=13] {\color{Grey}\begin{ytableau}
		{\color{black}4} & \none\\
		{\color{black}3} & \none\\
		{\color{black}1} & {\color{black}2}
		\end{ytableau}};
	\node[rectangle with rounded corners, draw=Grey, inner sep=3pt, minimum size=17pt]	(15) [below of=13] {\color{Grey}\begin{ytableau}
		{\color{black}3} & \none & \none\\
		{\color{black}1} & {\color{black}2} & {\color{black}4}
		\end{ytableau}};

	\path (1) edge   node {} (2)
	(1) edge   node {} (3)
	(1) edge   node {} (4)
	(2) edge   node {} (3)
	(2) edge   node {} (4)
	(2) edge   node {} (5)
	(2) edge   node {} (6)
	(2) edge   node {} (7)
	(2) edge   node {} (8)
	(3) edge   [bend right=37pt] node {} (4)
	(3) edge   node {} (6)
	(3) edge   node {} (7)
	(5) edge   node {} (6)
	(5) edge   node {} (8)
	(5) edge   node {} (9)
	(5) edge   node {} (10)
	(5) edge   node {} (11)
	(5) edge   node {} (12)
	(6) edge   node {} (7)
	(6) edge   [bend right=30pt] node {} (8)
	(8) edge   node {} (9)
	(8) edge   node {} (12)
	(9) edge   node {} (10)
	(9) edge   [bend right=27pt] node {} (11)
	(9) edge   node {} (12)
	(9) edge   node {} (13)
	(9) edge   node {} (14)
	(9) edge   node {} (15)
	(10) edge   node {} (11)
	(13) edge   node {} (14)
	(13) edge   node {} (15)
	(14) edge   [bend left=38pt] node {} (15)
	;
	\end{tikzpicture}
	\caption{The connected component of the standard element
$\pPr(1234)$ of $\belr_4$, omitting the loops at each vertex for clarity of the
picture.}
	\label{fig:connected_component}
\end{figure}
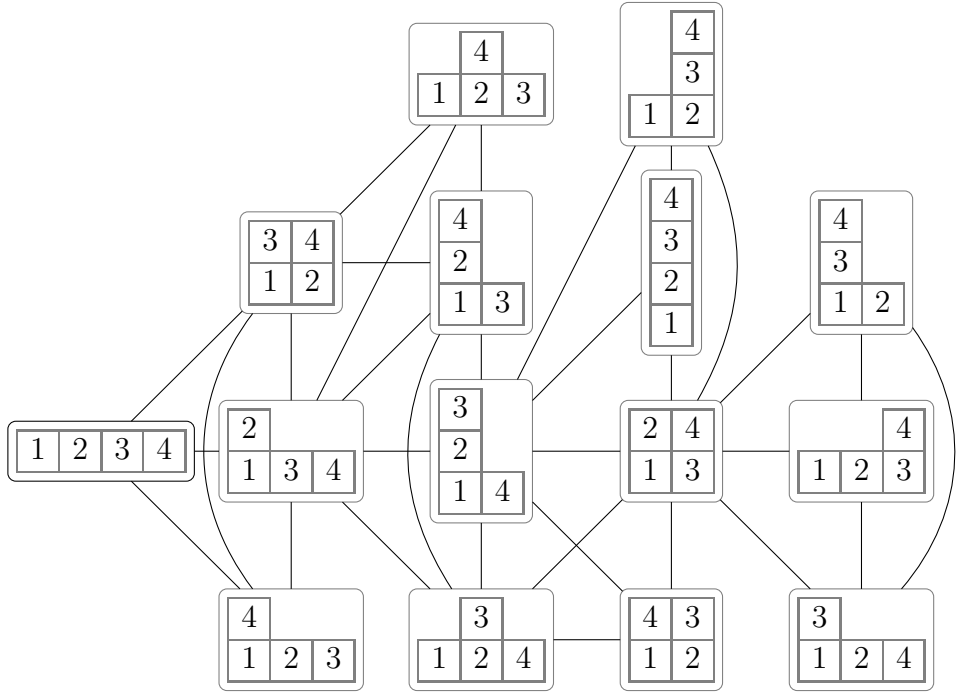

The results of computer experimentation on the diameter of connected
compontents is shown in Tables \ref{tab:std_cyclic_shift_graph} and
\ref{tab:cyclic_shift_graph}. In
Table~\ref{tab:std_cyclic_shift_graph} we present the diameter and
number of vertices in the connected component of the cyclic shift
graph of  standard elements of lengths $1$ up to $9$, whereas in
Table~\ref{tab:cyclic_shift_graph} the same information is presented
but for some (non-standard) words  of given fixed evaluations.

\begin{table}[htbp]
	\centering
	\caption{Examples of diameter and number of vertices in the connected
		component of the cyclic shift graph $\K(\belr)$ for given
evaluations
		of standard elements.}
	\begin{tabular}{ccccr}
		\toprule    \multicolumn{1}{p{5.57em}}{Length of
\newline{}standard \newline{}word} &
\multicolumn{1}{p{5.57em}}{Number of \newline{}vertices
in\newline{}connected\newline{}component} & \multicolumn{1}{p{5.57em}}{Diameter
of\newline{}connected\newline{}component} & \multicolumn{1}{p{6em}}{Diameter
as\newline{}a function of\newline{}word length} &  \\
		\toprule    1     & 1     & 0     & $n-1$   &  \\
		\midrule    2   & 2     & 1     & $n-1$   &  \\
		\midrule    3 & 5     & 2     & $2n-4$ &  \\
		\midrule    4 & 15    & 4     & $2n-4$ &  \\
		\midrule    5 & 52    & 6     & $2n-4$ &  \\
		\midrule    6 & 203   & 8     & $2n-4$ &
\\
		\midrule    7 & 877   & 10    & $2n-4$ &
 \\
		\midrule    8 & 4140  & 12    &
$2n-4$ &  \\
		\midrule    9 & 21147 & 14    &
$2n-4$ &  \\
		\bottomrule    \end{tabular}%
	\label{tab:std_cyclic_shift_graph}%
\end{table}%

The results in Table~\ref{tab:std_cyclic_shift_graph}
suggest the following:
\begin{conj}\label{conj:diameter_std}
The diameter of a connected component of $\K(\belr)$ containing a standard
element of length $n\geq 3$ is $2n-4$.
\end{conj}
Note that the connected components of  $\K(\belr)$ and $\K(\bell)$ coincide
when restricted to standard elements.

The data gathered in both Table~\ref{tab:std_cyclic_shift_graph} and
Table~\ref{tab:cyclic_shift_graph} leads us to propose the following:
\begin{conj}\label{conj:diameter_nonstd}
	The diameter of a connected component of $\K(\belr)$  containing an
element with  $n\geq 3$ symbols, with possible multiple appearences of
each symbol,  lies between	$n-1$ and $2n-4$.
\end{conj}

\begin{table}[htbp]
	\centering
	\caption{Examples of diameter and number of vertices in the connected
		component of the  cyclic shift graph $\K(\belr)$ for given
evaluations of
		non-standard elements.}
	\begin{tabular}{cccc}
		\toprule
		Evaluation & \multicolumn{1}{p{5.5em}}{Number of \newline{}vertices in\newline{}connected\newline{}component} & \multicolumn{1}{p{5.355em}}{Diameter of\newline{}connected\newline{}component} & \multicolumn{1}{p{6.145em}}{Diameter as\newline{}a function of\newline{}evaluation\newline{}length} \\
		\toprule
		(5)     & 1     & 0     & $n-1$ \\
		\midrule
		(5,3)   & 4     & 1     & $n-1$ \\
		\midrule
		(4,1,4) & 20    & 2     & $n-1=2n-4$ \\
		\midrule
		(3,3,1,2) & 75    & 3     & $n-1=2n-5$ \\
		\midrule
		(1,2,4,2) & 287   & 4     & $n=2n-4$ \\
		\midrule
		(1,3,2,1,2) & 656   & 5     & $n=2n-5$ \\
		\midrule
		(2,1,1,2,3) & 554   & 4     & $n-1=2n-6$ \\
		\midrule
		(1,2,1,2,2) & 711   & 6     & $n+1=2n-4$ \\
		\midrule
		(1,1,1,3,1,2) & 2409  & 7     & $n+1=2n-5$ \\
		\midrule
		(1,1,2,2,1,2) & 2840  & 6     & $n=2n-6$ \\
		\midrule
		(1,2,1,1,2,2) & 2373  & 8     & $n+2=2n-4$ \\
		\midrule
		(1,1,1,1,2,1,2) & 6499  & 9     & $n+2=2n-5$ \\
		\midrule
		(1,1,1,2,1,1,2) & 6078  & 8     & $n+1=2n-6$ \\
		\midrule
		(1,1,1,1,1,2,2) & 6768  & 10    & $n+3=2n-4$ \\
		\midrule
		(1,1,1,1,1,2,1,1) & 11695 & 11    & $n+3=2n-5$ \\
		\midrule
		(1,1,1,1,2,1,1,1) & 11224 & 10    & $n+2=2n-6$ \\
		\midrule
		(1,1,1,1,1,1,2,1) & 12002 & 12    & $n+4=2n-4$ \\
		\bottomrule
	\end{tabular}%
	\label{tab:cyclic_shift_graph}
\end{table}

One of the first results that was possible to obtain from the data was
\begin{lem}
	\label{prop:diameter12}
	All elements of $\belr$ containing two
	symbols,  with the same evaluation, form a connected
component of $\K(\belr)$. Furthermore, the component has
diameter $1$.
\end{lem}
\begin{proof}
	As already noticed each connected component of
$\K(\belr)$ cannot contain elements with different
	evaluations.
	Let $u$ and $v$ be two elements of $\belr$ with the same
evaluation
	such that $\abs{\cont(w)}= 2$. Suppose without loss of generality that
	$\cont(w)=\{1,2\}$. Then, these elements are of the form $\pPr(2^i1^j2^k)$,
	for some $i, k\in \mathbb{N}_0$ and $i+k,j\in\mathbb{N}$. So,
$u=\pPr(2^i1^j2^k)$
	and $v=\pPr(2^l1^n2^m)$ with
	$j=n$ and $i+k=l+m$. Therefore, $v=\pPr(2^l1^j2^m)$ and we consider the following
	cases:

	If $i\geq l$, then $k+i-l=m$. Setting $x=\pPr(2^{i-l})$ and $y=\pPr(2^l1^j2^k)$,
	we have
	\begin{align*}
	&u=\pPr(2^{i-l}2^l1^j2^k)=\pPr(2^{i-l})  \pPr(2^l1^j2^k)=x  y\,
\text{ and}\\
	&v=\pPr(2^l1^j2^k2^{i-l})=\pPr(2^l1^j2^k)  \pPr(2^{i-l})=y  x.
	\end{align*}

	Otherwise, if $i<l$, then $m+l-i=k$. Setting
	$x=\pPr(2^{i}1^{j}2^{m})$ and $y=\pPr(2^{l-i})$, we get
	\begin{align*}
	&u=\pPr(2^{i}1^{j}2^{m}2^{l-i})=\pPr(2^{i}1^{j}2^{m})
\pPr(2^{l-i})=x  y\, \text{ and}\\
	&v=\pPr(2^{l-i}2^{i}1^{j}2^{m})=\pPr(2^{l-i})
\pPr(2^{i}1^{j}2^{m})=y  x.
	\end{align*}
	In both cases, $u\psim v$. Therefore, the
	diameter of the connected component from $\K(\belr)$ containing such
	elements is 1. The result follows.
\end{proof}

In the following lemma we provide an upper bound for the diameter
of the connected components from $\K(\belr)$ of elements whose content
is greater or equal to $3$, thus answering the upper bound part of
Conjecture~\ref{conj:diameter_nonstd}.

By observing several connected components obtained with the program constructed
with SageMath, we
concluded that for any element $w\in \belr$, with $\cont(w)=\{1,\ldots,
n\}$ and $n\geq 3$, the element
\[\pPr\left((n-1)^{{\abs{w}}_{n-1}}
(n-2)^{{\abs{w}}_{n-2}}\cdots 3^{\abs{w}_3}
2^{\abs{w}_2}1^{\abs{w}_1}\ n^{{\abs{w}}_{n}}\right)\]
plays a key role in the connected component of $\K(\belr)$ which contains $w$.
For instance, in Figure~\ref{fig:connected_component},
we see that the element
\[\pPr(3214)={\color{Grey}\begin{ytableau}
{\color{black}3} & \none\\
{\color{black}2} & \none\\
{\color{black}1} & {\color{black}4}
\end{ytableau}}\]
is in the center of the connected component. Using this insight we were able
to prove the following result:

\begin{lem} \label{prop:diameter_upper_bound}
All elements of $\belr$ containing $n\geq 3$ symbols,  with the same
evaluation, form a connected
component of $\K(\belr)$. Furthermore, the component has
diameter  at most $2n-4$.
\end{lem}
\begin{proof}
	Let $w$ be an element of $\belr$ with $\abs{\cont(w)}=n\geq 3$.
	Suppose without loss of generality that $\cont(w)=\{1,\ldots, n\}$.
	Since each connected component of $\K(\belr_n)$
	cannot contain elements with different evaluations, to prove this
	result, it suffices to check that from $w$, by applying at most $n-2$
cyclic
shift relations we can always
	obtain the element
	\[w'=\pPr\left((n-1)^{\abs{w}_{n-1}} (n-2)^{
	\abs{w}_{n-2}}\cdots 2^{\abs{w}_2}1^{\abs{w}_1}
	n^{\abs{w}_{n}}\right)\]
	of $\belr$.

	We will construct a path in $\K(\belr_n)$ from $w$ to $w'$ of length at
most $n-2$. We aim to find a sequence $w_0,w_1,\ldots,w_{n-2}$ of elements of
$\belr_n$ such that $w=w_0$, $w'=w_{n-2}$, and $w_i\psim w_{i+1}$, for
$i=0,\ldots, n-3$. The construction is inductive.
	First note that all the symbols $1$ occur in the bottom of the first column
	of $w$. If $w$ has only one column, then $w$ has column reading
	$n^{\abs{w}_{n}}(n-1)^{\abs{w}_{n-1}}
(n-2)^{\abs{w}_{n-2}}\cdots 2^{\abs{w}_2}1^{\abs{w}_1}$ and applying one cyclic
shift we get the intended result. Suppose $w$ has at least two columns.
	Let $k$ (necessarily $k\geq 2$) be the bottom symbol of the
second column of $w$. Observe that any symbol $j$ less than $k$ must lie in the
first column of $w$. Set $w=w_0=\dots =w_{k-1}$.
	We calculate the element $w_k$ from $w$ in the following way. Consider
the   column reading $ukv$, of $w$, for $u,v\in \aA_n^*$, where  $u$ is the
prefix just up to before the first
occurrence of a symbol
		$k$ occurring in the second column. Fix
$w_k=\pPr(kv) \pPr(u)$. Note that $w=\pPr(u) \pPr(kv)$ and so $w\psim
w_k$.

	Then, the first column of $w_k$ has column reading
$k^{\abs{w}_k}\dots 2^{\abs{w}_2}1^{\abs{w}_1}$, because all symbols in $v$ are
greater or equal to $k$, and symbols in $u$ that are strictly less than $k$
appear in decreasing order.
	For $i\in\{k, \ldots, n-2\}$, let $w_{i}=\pPr\left(u\right) \pPr
	\left((i+1)v\right)$ where $u$ is the prefix of the column reading of
$w_{i}$
	up to just before the first occurrence of a symbol $i+1$ (in the second
column) and let $w_{i+1}=\pPr
	\left((i+1)v\right)  \pPr\left(u\right)$.
	Using this process we ensure that the first column of $w_{i+1}$ is
precisely \[\pPr\left((i+1)^{\abs{w}_{i+1}} i^{\abs{w}_i}
	\cdots 2^{\abs{w}_2} 1^{\abs{w}_1}\right).\]
	The result follows by induction.
\end{proof}

Regarding the lower bound of Conjecture~\ref{conj:diameter_nonstd}, we are only
able to establish it for standard elements
of $\K(\belr)$. To prove such a result, we will use the notion
of cocharge sequence for standard words over $\aA$ and follow an approach
similar to the one used in the case of the plactic monoid in \cite{1709.03974}.
 Note that it will be sufficient to prove the result for standard
words over the alphabet $\aA_n$.

For any standard word $w$ over $\aA_n$, the \emph{cocharge labels} from the
symbols
of $w$ are calculated as follows:
\begin{itemize}
	\item draw a circle, place a point $*$ somewhere on its circumference, and,
	starting from $*$, write $w$ anticlockwise around the circle;
	\item let the cocharge label of the symbol $1$ be $0$;
	\item iteratively, suppose the cocharge label of the symbol $a$
	from $w$ is $k$, then proceed clockwise from the symbol $a$ to
	the symbol $a+1$ and:
	\begin{itemize}
		\item if the symbol $a+1$ of $w$ is reached without passing the
		point $*$, then the cocharge label of $a+1$ is $k$;
		\item otherwise, if the symbol $a+1$ is reached after passing the
		point $*$, then the cocharge label of $a+1$ is $k+1$.
	\end{itemize}
\end{itemize}
The \emph{cocharge sequence} of  a standard word $w$, $\coch(w)$,
is the sequence of the cocharge labels from the symbols of $w$, whose
$a$-th term is the cocharge label from the symbol $a$ of $w$. So, it
follows from the definition that if $w$ is a standard word over
$\aA_n$, then $\coch(w)$ is a sequence of length $n$.

For example, the labelling of the standard word $w=4572631$ over
$\aA_7$, proceeds in the following way
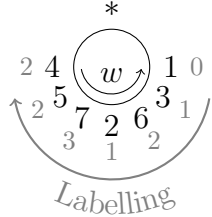
\begin{center}
\begin{tikzpicture}

\draw (0,0) circle[radius=5mm];

\draw[->] (4:-4mm) arc[radius=4mm,start angle=-170,end
angle=-10];

\draw[decorate,decoration={text along
	path,text={$w$},text align={align=center}}]
(-90:2mm);

\draw[gray,thick,->] (-18:13.5mm) arc[radius=13mm,start angle=-10,end
angle=-170];

\draw[gray,decorate,decoration={text along
	path,text={Labelling},text color=gray,text align={align=center}}]
(-165:18mm) arc[radius=18mm,start angle=-170,end angle=-10];

\foreach\i/\ilabel in {0/*,9/4,8/5,7/7,6/2,5/6,4/3,3/1} {
	\node at ($ (90-\i*30:7.8mm) $) {$\ilabel$};
};
\foreach\i/\ilabel in {9/2,8/2,5/2,7/3,3/0,6/1,4/1} {
	\node[font=\footnotesize,gray] at ($ (90-\i*30:11.2mm) $) {$\ilabel$};
};
\end{tikzpicture}
\end{center}
and thus $\coch(w)=(0,1,1,2,2,2,3)$.

From the definition it also follows that the cocharge sequence is a
weakly increasing sequence which starts at $0$ and such that each of
the remaining terms is either equal to the previous term or greater
by $1$.

\begin{lem}\label{standardinvariant}
	For   standard words $u,v$ over $\aA_n$, if
$u\belrcong v$,
	then $\coch(u)=\coch(v)$.
\end{lem}
\begin{proof}
	It is enough to show that any two standard words
	over $\aA_n$ such that one is obtained from the other by applying a
	relation from $\rR_{\belr_n}$ have the same cocharge sequence.

	So, there is a factor $yu_i\cdots u_1x$ of one of the standard words,
	with $x,y,u_1,\ldots,u_i\in\aA_n$ and $x<y< u_1<\cdots <u_i$, that is
	changed to the factor $yxu_i\cdots u_1$ of the other. Given any symbol
	$a\in \aA_n\setminus\{1\}$, when applying such relation, the relative
	position between the symbols $a$ and $a-1$ is not changed. That is, if
	$a-1$ occurs to the right (resp. left) of $a$ in one of the standard
	words, then $a-1$ also occurs to the right (resp. left) of $a$ in the
	other. Thus, equal symbols of these standard words have the same
	cocharge label and therefore the cocharge sequence of these words is
	the same.
\end{proof}

Given a standard element $u$ of $\belr_n$ on $n$ generators, let $\coch(u)$ be
$\coch(w)$
for any word $w\in \aA^*_n$ such that $\pPr(w)=u$. Using the previous
lemma we conclude that $\coch(u)$ is well-defined.

\begin{lem}\label{prop:diameter_lower_bound}
	The diameter of a connected component of
	$\K(\belr_n)$, with $n\geq 2$, containing a standard element is at
least $n-1$.
\end{lem}
\begin{proof}
The case $n=2$ follows from Lemma~\ref{prop:diameter12}. Suppose $n\geq 3$.

	From \cite[Lemma~2.2]{1709.03974}, we deduce
	that any two standard elements of $\belr$ that differ by a single
cyclic
	shift, have cocharge sequences whose corresponding terms differ by at most
	$1$.

	The standard elements $u=\pPr\left(1\cdot 2\cdots
n\right)$ and 	$v=\pPr\left(n(n-1)\cdots 1\right)$ of $\belr_n$ are in the
same connected
	component of $\K(\belr)$ by Lemma~\ref{prop:diameter_upper_bound}. Now,
notice that $\coch(u)=(0,0,\ldots,0)$ and
	that $\coch(v)=(0,1,\ldots,n-1)$. Since the last term of both sequences
	differs by $n-1$, the standard  elements $u$ and $v$ are
	at distance of at least $n-1$.
\end{proof}

For instance, in Figure~\ref{fig:connected_component}, the distance between the
elements
\[\pPr(1234)={\color{Grey}\begin{ytableau}
	{\color{black}1} & {\color{black}2} & {\color{black}3} & {\color{black}4}
	\end{ytableau}}\ \text{ and }\ \pPr(4321)={\color{Grey}\begin{ytableau}
	{\color{black}4} \\
	{\color{black}3} \\
	{\color{black}2} \\
	{\color{black}1}
	\end{ytableau}}\]
in that connected component is precisely $3$, which is in accordance with
the previous result.

Since for standard words the cyclic shif graphs $\K(\bell)$ and $\K(\belr)$
coincide, the previous result also give us a lower bound for
connected components of standard words of $\K(\bell)$.

Combining Lemmata~\ref{prop:diameter12}, \ref{prop:diameter_upper_bound} and
\ref{prop:diameter_lower_bound}
we get
\begin{thm}
 \begin{enumerate}
  \item Connected components of $\K(\belr)$ coincide with $\evsim$-classes of
$\belr$.
  \item The maximum diameter of a connected component of $\K(\belr_n)$ is
$n-1$, for $n=1,2$, and lies between $n-1$ and $2n-4$, for $n\geq 3$.
 \end{enumerate}
\end{thm}

Other observations from computer experimental results
 lead us to conclude that the number of
vertices in a given connected component is equal to the number of vertices in
the connected component that has one more symbol $1$. This makes sense since
the elements of the new connected component will be the elements of the
former with an additional symbol 1 in the bottom of the first column.

Also, it seems that in a standard component, the addition of a new symbol $1$
leads to
a connected component whose diameter can possibly decrease by 2 when
compared with the original. In fact,
we were able to establish the following result:

\begin{lem}
	Let $w$ be an element of $\belr$, with $n\geq 4$ symbols, such that the
minimum symbol of $w$ has at least two occurences, and the second smallest
symbol only occurs once. Then the diameter
	of the connected component of $\K(\belr)$ containing $w$ is at most
$2n-6$.
\end{lem}
\begin{proof}
Without lost of generality, suppose that  $\cont(w)=\{1,\dots ,n\}$, with
$n\geq 4$.
The proof strategy is similar to the proof of
Lemma~\ref{prop:diameter_upper_bound}. We aim to
construct a path in $\K(\belr)$ from $w$ to
\[w'=\pPr\left(1^{\abs{w}_1}
	(n-1)^{\abs{w}_{n-1}}(n-2)^{\abs{w}_{n-2}}\cdots 3^{\abs{w}_3}2
n^{\abs{w}_n}\right)\]
by applying at most $n-3$ cyclic shifts relations.

For an element $w$ of $\belr$, under the given assumptions, we will distinguish
 particular readings of its tableau representation. For simplicity, we call
these readings  \emph{delayed
column readings}. Note that the symbol $1$
occurs more than once, and that all symbols $1$ appear on the bottom of the
first column of such tableaux. If we proceed as in the column reading, but we
read the symbol on the bottom of the first column (necessarily a symbol $1$)
latter on,  we obtain a delayed column reading. Following
Algorithm~\ref{alg:PSinsertion}, it is clear that all these words corresponding
to delayed column readings also insert to the same element. For example,
the element $S$ of \eqref{exmp2} has column reading $411\, 5432$ and has
delayed column readings, $4151432$, $4154132$, $4154312$, and $4154321$.

If the tableau representation of $w$ has only one column, then it has the form
\[\pPr\left(
	n^{\abs{w}_n}(n-1)^{\abs{w}_{n-1}}(n-2)^{\abs{w}_{n-2}}\cdots
3^{\abs{w}_3}2
1^{\abs{w}_1}\right)\]
which is cyclic shift related to
\[\pPr\left(
	(n-1)^{\abs{w}_{n-1}}(n-2)^{\abs{w}_{n-2}}\cdots
3^{\abs{w}_3}2
1^{\abs{w}_1}n^{\abs{w}_n}\right)\]
which in turn has delayed column reading
\[(n-1)^{\abs{w}_{n-1}}(n-2)^{\abs{w}_{n-2}}\cdots
3^{\abs{w}_3}2
1^{\abs{w}_1-1}n^{\abs{w}_n}1.\]
By applying a cyclic shift we get the intended form since
\[w'= \pPr\left(
1(n-1)^{\abs{w}_{n-1}}(n-2)^{\abs{w}_{n-2}}\cdots
3^{\abs{w}_3}2
1^{\abs{w}_1-1}n^{\abs{w}_n}\right).\]

Otherwise, suppose first that the bottom symbol of the second column is $2$.
Note that the symbol $3$ can appear in the first three columns of $w$, and if
it appears in the third column, then its bottom symbol is a $3$. Consider the
delayed column reading of $w$, $u13v$, where $u$ is the prefix up to before the
first occurence of a symbol $3$ in the rightmost column where a symbol $3$
appears (necessarily on the first three columns). So, either $u$ or $v$ has the
unique symbol $2$, and if $u$ or $v$ has the symbol $2$ then all symbols $3$
appear to its left.
Let $w_3=\pPr(13v) \pPr(u)$, and so $w\psim
w_3$, since $w=\pPr(u) \pPr(13v)$. The first column of $w_3$ has
column reading $1^{\abs{w}_1}$ and the
second column $3^{\abs{w}_3}2$.

Now suppose the bottom symbol of the second column is $k>2$. Consider the
delayed column reading of $w$, $u1kv$, where $u$ is the prefix up
to before the first occurence of a symbol $k$ in the second column. Note that
all symbols in $v$ are
greater or equal to $k$, and symbols in $u$ that are strictly less than $k$
appear in decreasing order (from left to right).
Let $w_k=\pPr(1kv) \pPr(u)$, and so $w\psim
w_k$. The first column of $w_k$ has column reading $1^{\abs{w}_1}$ and the
second column $k^{\abs{w}_{k}}\cdots
3^{\abs{w}_3}2 $.

We will construct a path in $\K(\belr_n)$ from $w_k$ to $w'$ of length at
most $n-4$, by considering a sequence $w_k,\ldots,w_{n-1}$ of elements of
$\belr_n$, with $k\geq 3$, such that $w'=w_{n-1}$, and $w_i\psim w_{i+1}$, for
$i=k,\ldots, n-1$. For $i\in\{k, \ldots, n-2\}$, let
$w_{i}=\pPr\left(u\right) \pPr
	\left(1(i+1)v\right)$ where $u$ is the prefix of the delayed column
reading $u1(i+1)v $ of $w_{i}$
	up to just before the first occurrence of a symbol $i+1$ (on the third
column) and let $w_{i+1}=\pPr\left(1(i+1)v\right)  \pPr\left(u\right)$.
Note that all symbols in $v$ are greater or equal to $i+1$, and all symbols in
$u$ that are strictly less than $i+1$  appear in decreasing order (from left to
right).
	Thus the two first columns of $w_{i+1}$ have column readings
$1^{\abs{w}_1}$ and $(i+1)^{\abs{w}_{i+1}} i^{\abs{w}_i}\dots
3^{\abs{w}_3}2$, respectively.
	The result follows by induction.
\end{proof}

\section{Conjugacy in the \lps\ and \rps\ monoids}
\label{subsection:conjugacy}

Restating the results of Section~\ref{subsection:cyclic_shift} in terms of the
conjugacy relation $\psim$ we have shown that in $\belr_n$  we have
${\psim}={\evsim}$,  for
$n\in\{1,2\}$; and that ${\psim}\subsetneq {\tpsim}={\evsim}$, for $n>2$. Thus,
${\tpsim}={\evsim}$ in the (infinite rank) right Patience Sorting monoid.
In all cases, we deduce that any of the conjugacy relations ${\tpsim}$,
${\osim}$, and ${\lsim}$ coincides with ${\evsim}$.

The \rps\ case proves to be distinct from the \lps\ case. In $\bell_1$, it is
immediate that ${\psim}={\evsim}$, but for $n\geq 2$,  we will see that
${\psim}\subsetneq
{\tpsim}$ and
${\lsim}\subsetneq {\evsim}$, in $\bell_n$, and thus in $\bell$. Whether the
inclusion
${\tpsim}\subseteq {\lsim}$ is strict
or, in fact an equality, is left as an open question.

\begin{prop}
	For any $n\geq 2$, in $\bell_n$ we have ${\psim}\subsetneq {\tpsim}$.
\end{prop}
\begin{proof}
	From Lemma~\ref{prop:diameter_lower_bound}, we deduce that ${\psim}
	\subsetneq {\tpsim}$, for $\bell_n$ with $n\geq 3$.

	Regarding the $\bell_2$ case, consider the elements $\pPl(21121)$ and
	$\pPl(21112)$ of $\bell_2$. We have that \begin{align*}
	\pPl(21121)&=\pPl(211) \pPl(21)\psim
\pPl(21) \pPl(211)=\pPl(21211)\\
	&=\pPl(22111)=\pPl(2) \pPl(2111)\psim \pPl(2111) \pPl(2)\\
	&=\pPl(21112),
	\end{align*}
	and so $\pPl(21121)\ \tpsim\ \pPl(21112)$ in $\bell_2$. It is easy to check that
	$\pPl(21121)\npsim \pPl(21112)$ in $\bell_2$. Indeed, notice that the unique words
	$u$ and $v$ of $\aA^*_2$ such that $\pPl(u)=\pPl(21121)$ and $\pPl(v)=\pPl(21112)$
	are precisely, $21121$ and $21112$, respectively. Moreover, if
$\pPl(21121)=\pPl(st)$, for
	words $s,t\in\aA_2^*$, then	$\pPl(ts)\neq \pPl(21112)$.

	Resuming, we have a pair of elements of $\bell_2$ which belong to $\tpsim$
	but not to $\psim$.
\end{proof}

In order to prove that ${\lsim}\subsetneq {\evsim}$, in $\bell_n$, we first
prove two auxiliary
results.

\begin{lem}
For any $k,n\in \mathbb{N}$ and $u,v\in \bell_k$, if $n\geq k$, then:
\begin{equation*}
	u\lsim v \text{ in } \bell_n\ \Leftrightarrow\ u\lsim v \text{ in }
	\bell_k.
\end{equation*}
\end{lem}
\begin{proof}
Let $u,v\in \bell_k$ and $n\geq k$. Suppose that $u\lsim v$ in $\bell_n$.
Note that $u$ and $v$ have the same evaluation.
 There exists $g\in \bell_n$
such that $u g= g v$. If $g$ is the identity then the result holds trivially.
Assume that the tableau representation of
$g$ has $j$ columns.

Since $u g= g v$, then $u^2g=u u g
= u g v= g v v=g v^2$. Using the same reasoning, it
follows
that for any $i\geq 1$, $u^i g=g v^i$.
Note that if $a$ is the minimum symbol occuring in $u$, then $u^i$ has bottom
row beginning (from left to right) with (at least) $i$ symbols $a$.

Suppose $g$ has a symbol greater than $k$.
As $\cont(u)\subseteq \aA_k$, the symbols from ${g}$ that are greater or
equal than $k$ have to be inserted in the tableau representation of $u^j$ to
the right of the first $j$ columns. Now, in the tableau representation of
$gv^i$, the symbols from $g$ are inserted into the first $j$ columns. This is a
contradiction, since $u^ig=gv^i$. So all symbols from $g$ are less or equal
than $k$, that is, $g\in\bell_k$.

The converse direction of the lemma is obvious from the definition of
$\lsim$.
\end{proof}

Let $C_2=\left\{\pPl(1),\pPl(21)\right\}$.
As proved in \cite[Proposition~4.1]{1706.06884}, the submonoid of $\bell_2$
generated by $C_2$, denoted $\langle C_2\rangle$, is free. Observe that the
elements of  $\langle C_2\rangle$ are precisely the elements of $\bell_2$ whose
tableau representation has bottom row filled with symbols $1$.

\begin{lem}\label{lem:left_conjugacy_Bell2}
For any $u,v\in \langle C_2\rangle$ and $n\geq 2$,
\begin{equation*}
	u\lsim v \text{ in } \bell_n\ \Leftrightarrow\ u\lsim v
	\text{ in } \langle C_2\rangle.
\end{equation*}
\end{lem}
\begin{proof}
Let $u,v\in \langle C_2\rangle$, $n\geq 2$ and suppose that
$u\lsim v$ in $\bell_n$.
Suppose that $u\in\langle \pPl(21)\rangle$. Since $u\evsim v$,  then also $v\in
\langle\pPl(21)\rangle$, and thus $u=v$. Therefore the result holds.

Suppose now that $u\notin \langle \pPl(21)\rangle$. Then at least one of the
columns of the tableau representation of $u$ has height one and is filled
with the symbol $1$. Note that the tableau representation of $v$ has the same
number of columns of heigth two, and the same number of columns of
heigth one (and each such box is filled with the symbol $1$).

Let $g\in \bell_n$ be such that $u g= g v$.
By the previous lemma we can assume $g\in\bell_2$.
If $g$ is the identity then the result holds trivially.
Suppose that the tableau representation of $g$ has  at
least one  column with height one filled with the symbol $2$.
Attending to Algorithm~\ref{alg:PSinsertion} and since the bottom row of $u$
is filled with  the symbol $1$, $ug$ is represented by a tableau that is
 composed by the columns of $u$ followed by the columns of $g$.

Now,
the tableau representation of $gv$ has at least one less column. Indeed,
consider the column reading of the tableau representation of $v$, which is a
word from $\{1,21\}^*$, where at least one single symbol $1$ is used, that is,
it does not belong to $\{21\}^*$. Applying Algorithm~\ref{alg:PSinsertion} we
will first insert symbols from $g$, and get the tableau representation of $g$,
followed by the insertion of the column reading from $v$. Now, each time a word
$21$ is inserted we obtain a new column, but the first time a single symbol
$1$ is inserted it will take place in the leftmost column of height one filled
with the symbol $2$, becoming a column of heigth two and column reading $21$.
Thus, the tableau representation of $gv$ cannot have the same number of columns
as the tableau representation of $ug$. This is a contradiction. Therefore, the
tableau representation of $g$ has bottom row filled with the symbol $1$, and
hence $g\in\langle C_2\rangle$.

Since the converse direction is immediate, the result follows.
\end{proof}

\begin{prop}
	\label{prop511}
	For the \lps\ monoid of rank $n$, with $n\geq 2$, we have
	\begin{equation*}
		\lsim\ \subsetneq\
\evsim.
	\end{equation*}
\end{prop}
\begin{proof}
In the free monoid of rank $2$  the relation $\tpsim$ is equal to $\lsim$
\cite[Theorem~3]{lentin1967combinatorial}, and it is properly contained in
$\evsim$ (For example, in $\aA_2^*$, there are words with the same
evaluation $2121$ and $2112$, for which $2121\nlsim 2112$).

Consider the embedding  $\eta:\aA_2^*\to\bell_n$ given by $1\mapsto
\pPl(1)$ and $2\mapsto \pPl(21)$. This map yields an isomorphism between
$\aA_2^*$ and the free submonoid of $\bell_n$, $\langle C_2\rangle$. Using the
example of the first paragraph and the isomorphism, we conclude that the
elements $\pPl(211211)$ and $\pPl(211121)$ of $\langle C_2\rangle$ that have
the same evaluation,  satisfy $\pPl(211211) \nlsim \pPl(211121)$ in
$\langle
C_2\rangle$. By  Lemma~\ref{lem:left_conjugacy_Bell2} we get $\pPl(211211)
\nlsim \pPl(211121)$ in $\bell_n$.  The result follows.
\end{proof}

Regarding the relation between $\tpsim$ and $\lsim$ in the \lps\ monoids
of rank greater or equal than $3$ we leave the following:
\begin{prob}
In any multihomogeneous monoid the inclusion $\tpsim\ \subseteq\ \lsim$
holds. For the \lps\ monoid of rank $n$, $\bell_n$, with $n\geq 3$, is the
inclusion strict, or does the equality hold?
\end{prob}

Considering this problem we were able to prove the following
result:
\begin{prop}
Let $u,v$ be elements of $\bell_n$ with exactly two symbols (with possible
multiple occurrences)  and  $n\geq 2$. In $\bell_n$, the following holds
		\begin{equation*}
		u\ \tpsim\ v\ \Leftrightarrow\ u\ \lsim\ v.
		\end{equation*}
\end{prop}
\begin{proof}
Without lost of generality, assume that  $u,v\in \bell_2$ and that
$u\lsim v$ in $\bell_n$. Hence
$u \evsim v$ and thus for $a\in\aA_2$, the number of symbols $a$
in $u$ and $v$ is the same.

As $u,v\in \bell_2$, $u=\pPl(u'u'')$ and $v=\pPl(v'v'')$ where $\pPl(u'),\pPl(v')
\in\langle C_2\rangle$, and $\pPl(u''),\pPl(v'')\in \langle\pPl(2)\rangle$.
Note that
$u\psim \pPl(u''u')$ and $v\psim \pPl(v''v')$ in $\bell_n$.

We consider two cases. If $\abs{u'u''}_2\geq \abs{u'u''}_1$, then
$\pPl(u''u')=\pPl\left((21)^i2^j\right)$ and $\pPl(v''v')
=\pPl\left((21)^k2^l\right)$ for some $i,j,k,l\in\mathbb{N}_0$.
As $\abs{u''u'}_a=\abs{v''v'}_a$ for all $a\in\aA_2$, we deduce that
$i=k$ and $i+j=k+l$, and thus it follows that $j=l$. So, we conclude that
$\pPl(u''u')=\pPl(v''v')$. Therefore $u\psim \pPl(u''u')=\pPl(v''v')\psim v$
and thus $u\ \tpsim\ v$ in $\bell_n$.

Now suppose that ${|u'u''|}_1>{|u'u''|}_2$. In this case $\pPl(u''u'),
\pPl(v''v')\in \langle C_2\rangle$. As in $\bell_n$
$\pPl(u''u') \psim u$, $u\lsim v$, $\pPl(v''v')\psim v$ and ${\psim}\subseteq
{\lsim}$, it  follows that $\pPl(u''u')\lsim \pPl(v''v')$ in $\bell_n$, by the
transitivity of $\lsim$. Hence, by Lemma~\ref{lem:left_conjugacy_Bell2},
$\pPl(u''u')\lsim \pPl(v''v')$ in the free monoid  $\langle C_2\rangle$. In a free
monoid we have ${\tpsim} ={\lsim}$  \cite[Theorem~3]{lentin1967combinatorial}.
Therefore $\pPl(u''u')\ \tpsim\ \pPl(v''v')$ in $\langle C_2\rangle$. So,
$\pPl(u''u')\ \tpsim\ \pPl(v''v')$ in $\bell_n$. Combining this with fact
that $u\psim \pPl(u''u')$ and $\pPl(v''v')\psim v$ in $\bell_n$, it follows
that $u\ \tpsim\ v$ in $\bell_n$.

In both cases $u\ \tpsim\ v$ in $\bell_n$ and the result follows.
\end{proof}

\bibliographystyle{alpha}
\bibliography{\jobname}

\end{document}